\documentclass[11pt,a4paper]{amsart}

\usepackage{amssymb}
\usepackage{url}
\usepackage{mathrsfs}
\usepackage{hhline}
\usepackage{multirow}
\usepackage{xcolor}
\usepackage[title]{appendix}

\newcommand{\AS}{ A$^s$MDS }
\newcommand{\AT}{ A$^t$MDS }
\newcommand{\F}{\mathbb{F}}
\newcommand{\K}{\mathcal{K}}

\newcommand{\G}{\mathcal{G}}

\theoremstyle{plain}
\newtheorem{theorem}{Theorem}[section]
\newtheorem{proposition}[theorem]{Proposition}
\newtheorem{lemma}[theorem]{Lemma}
\newtheorem{corollary}[theorem]{Corollary}
\newtheorem{conjecture}[theorem]{Conjecture}
\theoremstyle{definition}
\newtheorem{definition}[theorem]{Definition}
\newtheorem{example}[theorem]{Example}
\theoremstyle{remark}
\newtheorem{remark}[theorem]{Remark}

\begin{document}

	\title[Bounds on the length of codes of non-zero defect]{Projective systems and bounds on the length of codes of non-zero defect}

\author{Tim L. Alderson}
\address{Mathematics and Statistics, University of New Brunswick, Saint John, New Brunswick, E2L 4L5, Canada}
\email{tim@unb.ca}
\thanks{The first author acknowledges the support of the Natural Sciences and Engineering Research Council of Canada (NSERC), [funding reference number 2019-04103]. Cette recherche a \'{e}t\'{e} financ\'{e}e par le Conseil de recherches en sciences naturelles et en g\'{e}nie du Canada (CRSNG), [num\'{e}ro de r\'{e}f\'{e}rence 2019-04103].}

\author{Zhipeng Zhang}
\address{Mathematics and Statistics, University of New Brunswick, Saint John, New Brunswick, E2L 4L5, Canada}
\email{zhipeng.zhang@unb.ca}

\date{}

	\begin{abstract}
		We derive bounds on the lengths of linear codes with fixed Singleton defect $s$, working within the framework of projective systems as advocated by Tsfasman and Vl\v{a}du\c{t}. This geometric perspective allows us to unify and extend a range of existing results. We introduce the parameter $m^s(k,q)$, denoting the maximum length of a non-degenerate $[n,k,d]_q$ A$^s$MDS code, and more generally $m^s_t(k,q)$, where the dual code is additionally required to be A$^t$MDS. We also study $\kappa(s,q)$, the maximum dimension $k$ for which a length-maximal A$^s$MDS code exists. Among our main results, we provide sufficient conditions on $n$ and $k$ under which the dual of an A$^s$MDS code is necessarily A$^s$MDS, addressing a gap in the existing literature. We show that codes of sufficient length must be projective, meet the Griesmer bound, and be dual to an AMDS code. Our bounds subsume or improve several results in the literature. Two conjectures on the non-existence of length-maximal codes of dimension $k\ge 5$ are proposed, supported by computational evidence.
	\end{abstract}

	\maketitle

	\section{Introduction and preliminaries} \label{sec: Intro}

	This section establishes the notation and background used throughout. We begin by recalling the definitions of linear codes and projective systems, then introduce the central parameters $m^s(k,q)$, $m^s_t(k,q)$, and $\kappa(s,q)$, and survey existing bounds in the literature. Section~2 treats codes of small dimension, Section~3 reviews bounds from the theory of classical arcs, and Section~4 presents the main results. Section~5 concludes with open questions and conjectures.

	For $n\geq k$, a linear code $C$ of length $n$, dimension $k$, and minimum distance $d$, denoted an $[n,k,d]_q$ code, is a $k$-dimensional subspace of $\F_q^n$ such that the minimum (Hamming) distance between any two elements (codewords) of $C$  is $d$; that is to say, there exist two codewords agreeing in $n-d$ coordinates and no two codewords agree in as many as $n-d+1$ coordinates.  A linear code $C \subset \mathbb{F}_q^n$ is degenerate if $C \subseteq \mathbb{F}_q^{n-1} \subset \mathbb{F}_q^n$, where $\mathbb{F}_q^{n-1}$ is the subspace of vectors with $0$ at some fixed position, otherwise, we say that the code is non-degenerate. Unless otherwise stated, the codes considered here shall be non-degenerate. Given an \([n, k, d]_q\) code \( C \), the dual code \(C^\perp = \{ y \in \mathbb{F}_q^n \mid y\cdot x = 0 \text{ for all } x \in C \}\) is an $[n,k^\perp, d^\perp]_q$ code where $k^\perp = n-k$.

	Following Tsfasman and Vl\v{a}du\c{t} \cite{Tsfasman2007}, we represent linear codes as \textit{projective systems} --- finite multisets of points in a projective space --- thereby gaining a geometric vantage point from which to analyze code parameters. We now recall the necessary definitions.

	The projective space of $ k$ dimensions over $\F_q$ will be denoted by PG$(k, q)$. A projective subspace of dimension $m$ of PG$(k-1,q)$ is called an $m$-\textit{flat}; it is isomorphic to PG$(m,q)$ as an abstract projective space, though we retain the term flat to emphasize its embedding in the ambient space. We say that a set of $m$ points of PG$(k-1, q)$ are \textit{independent} if they are not contained in an $(m-2)$-flat. More generally, the quotient of $PG(k,q)$ by a $d$-flat $\Omega$ is isomorphic to PG$(k-d-1,q)$. The flats of the quotient space are precisely the flats of PG$(k,q)$ containing $\Omega$, (see e.g. \cite{Ball2015}).

	An $[n, k, d]_q$ code $C$ is determined by an associated generator matrix, $G$. Up to code equivalence, the columns  of  $G$ can be considered as an $n$-multiset $\G$ of points in $PG(k-1,q)$,  at most $n-d$ per hyperplane (subspace of co-dimension $1$). Any hyperplane meeting $\G$ in $n-d$ points is said to be a \textit{secant} of $\G$.   We may thus discuss codes in terms of equivalent objects known as \textit{projective systems} (see e.g. \cite{Tsfasman2007}).

	A projective $[n, k, d]_q$ system is a finite unordered family $\G$ of points  in $\Sigma =\text{PG}(k-1,q)$ that does not lie entirely in a  hyperplane  of $\Sigma$. Though generally, $\G$ may be a multiset, we shall employ a slight abuse of notation, writing  $\mathcal{G} \subseteq \Sigma$, and  $|\G|$ to denote the cardinality of $\G$ counting multiplicities. For a point $P$, we denote by $\mu(P)$ the multiplicity of $P$ in $\G$. The parameters $n, k,$ and $ d $ are defined as follows:
	\begin{equation}\label{eqn: Proj Sys Code parameters}
		n=|\G|, \quad k=\operatorname{dim} \Sigma +1, \quad n-d= \max \{|\G \cap H| \geq 1 \mid  H \text{ is a hyperplane of } \Sigma\}.
	\end{equation}

	Similarly, the set of parameters $[n,k^\perp,d^{\perp}]_q$ of a dual code $\G^\perp$ can  be characterized in terms of the projective system $\mathcal{G}$. For a projective system $\mathcal{G}$,
	\begin{equation}\label{eqn: Proj Sys dual distance}
		d^{\perp}=\min \{|\mathcal{Q}|: \mathcal{Q} \subset \G,|\mathcal{Q}|-\operatorname{dim} \operatorname{lin}\langle\mathcal{Q}\rangle=1\},
	\end{equation}
	where $\langle\mathcal{Q}\rangle$ is the span of $\mathcal{Q}$, and $\operatorname{dim} \operatorname{lin}\langle\mathcal{Q}\rangle$ is its linear dimension (which is greater by 1 than the projective dimension of $\langle\mathcal{Q}\rangle$ ). Equivalently, a set $\mathcal{Q}\subset\mathcal{G}$ is \textit{dependent} if $|\mathcal{Q}|-\operatorname{dim}\operatorname{lin}\langle\mathcal{Q}\rangle=1$, i.e.\ the points of $\mathcal{Q}$ fail to span a $\mathrm{PG}(|\mathcal{Q}|-1,q)$; thus $d^\perp$ is the minimum size of a dependent set of points of $\mathcal{G}$. Note that degenerate codes are degenerate projective systems, in that they contain  a point $0$, corresponding to the $0$-dimensional vector subspace $\langle 0 \rangle$. Though we generally avoid degenerate codes here, in such cases, the point $0$ is said to have projective dimension $-1$.  Notice that according to Eqn.~(\ref{eqn: Proj Sys dual distance}), a code $\G$ is degenerate if and only if $d^\perp=1$.

	Our working definition of a (non-degenerate) linear code is as follows.

	\begin{definition}\label{defn: linear Code by PS}
		A linear $[n,k,d]_q$ code $\G$ is an $n$-(multi)set of points in $\Sigma = $PG$(k-1,q)$ such that
		\[
		n-d= \max \{|\G \cap H| \geq 1 \mid  H \text{ is a hyperplane of } \Sigma\}.
		\]
		That is to say, $\G$ is an $[n,k,d]_q$ projective system.
	\end{definition}

	It follows immediately that  an $ [n,k,d ]_q$ code $\G$ satisfies $n-d\geq k-1$, which is the well known Singleton bound.  The \textit{Singleton defect} of an $[n,k,d]_q$ code is $S(\G)=n-k+1-d$.  An $ [n,k,d]_q$ code $\G$ with $S(\G)=0$ is called a maximum distance separable (MDS) code.  Codes of Singleton defect $1$ are called Almost-MDS (AMDS) codes \cite{MR1409442}, and more generally, codes with $S(\G) = s$ are denoted A$^s$MDS codes. The dual code of an MDS code is MDS, but generally speaking the dual of an A$^s$MDS code need not be A$^s$MDS.  If $S(\G)=S(\G^\perp)=1$, $\G$ is said to be dually-AMDS (or Near MDS, or NMDS). This notation extends naturally to dually A$^s$MDS (or N$^s$MDS) codes.

	A linear code  $\G$ is \textit{projective} if $\G$ is a set, that is to say if $P\in \G$ then $\mu(P)=1$. A non-projective code is therefore equivalent to a code with one or more repeated coordinates. As we shall see, codes of sufficient length are necessarily projective.

	Projective systems may be viewed as arcs or, sometimes as caps in PG$(k-1,q)$. An \textit{$(n,r)$-arc} $\K$ in $\Pi=PG(k-1,q)$, $k\leq r+1$, is a collection of  $n$ points such that each hyperplane of $\Pi$ is incident
	with at most $r$ points in $\K$, and some hyperplane is incident with $r$ points of $\K$. In the case where $\K$ is a multi-set, we shall use the term multi-arc. An $(n,k-1)$-arc in $PG(k-1,q)$ is an \textit{$n$-arc}.
	From the definitions, it is immediate that $(n,n-d)$-arcs in $PG(k-1,q)$ and  projective $[n,k,d]_q$-codes are equivalent objects, as are  $(n,n-d)$-multi-arcs in $PG(k-1,q)$ and $[n,k,d]_q$-codes. An $(n,r)$-arc in $PG(k-1,q)$ is said to be maximal if $n=(r-k+2)(q+1)+k-2$. Such arcs are complete, in that they are not contained in an $(n+1,r)$-arc. The (non)-existence question regarding maximal $(n,r)$-arcs in PG$(2,q)$ has largely been answered (see e.g.\cite{MR2158768} ). Here, for fixed $q$, and $s$ and with $r=k+s-1$, we  investigate the parameter $\kappa(s,q)$, representing the maximum value $k$ for which there exists a maximal $(n,r)$-arc in $PG(k-1,q)$.

	An $n$-cap in PG$(k-1,q)$ is an $n$-set of points such that no three are collinear. It follows immediately that for $k\ge 3$, a projective $[n,k,d]_q$ code, $\G$ is an $n$-cap in PG$(k-1,q)$ if and only if $d^\perp \ge 4$ (equivalently, $S(C^\perp)\le k-3$).

	A fundamental problem in coding theory is that of determining the maximum length of a code with $k$, $q$, and $S(\G)$ fixed. To this end, we make the following definition.

	\begin{definition}\label{def: m^s(k,q)}
		We shall denote by $m^{s}_t(k,q)$ the maximum length of a (non-degenerate) $[n,k,d]_q$ A$^s$MDS code $\G$ such that  $\G^\perp$ is an A$^t$MDS code. If $t$ is not restricted, we shall simply write $m^{s}(k,q)$. In keeping with the standard notation, in the special cases $s=0$, and $s=t=1$ these parameters shall be denoted $m(k,q)$, and $m'(k,q)$ respectively.
	\end{definition}

	Investigations into the bounds on $m^s_t(k,q)$ initially focused on codes of small defect. In particular, the case $s=0$ has received much attention by way of the Main Conjecture on MDS codes (see Section \ref{subsec: MDS Codes}). The focus of the current work is on codes of defect $s>0$.
	For $s=1$, De Boer \cite{MR1409442} and Dodunekov and Landjev \cite{MR1755411,MR1358272} provide the following.

	\begin{proposition}[\cite{MR1409442}, Thm. 8; \cite{MR1358272}, Prop. 6.2]  \label{prop: AMDS de Boer}
		If $q>3$ and $k\ge 3$ then $m^1(k,q)\le 2q+k-2$.
	\end{proposition}

	Recall the notation $m'(k,q):= m^1_1(k,q)$.
	\begin{theorem}[\cite{MR1755411}, Thm. 2.7] \label{thm: Dodunekov 2.7} Let $k\ge 2$.
		\begin{enumerate}
			\item  $m'(k,q) \le 2q+k$; \label{part: thm Dodunekov P1}
			\item $m'(2,q) = 2q+2$; \label{part: thm Dodunekov P2}
			\item $m'(k,q)\le m'(k-\alpha, q)+\alpha$  ; for every $\alpha$, $0\le \alpha \le k$; \label{part: thm Dodunekov P3}
			\item  $m'(k,q)=k+1 $ for $k>2q$; \label{part: thm Dodunekov P4}
			\item  $m'(2q,q)=2q+2$ for $q>3$. \label{part: thm Dodunekov P5}
			\item $m'(2q-1,q)=2q+1$ for $q>3$. \label{part: thm Dodunekov P6}
		\end{enumerate}
	\end{theorem}

	Dodunekov and Landjev furthermore showed that an $(n,k,d)_q$ AMDS code with $n>q+k$ must in fact be NMDS. Written using our current notation, their result is as follows.
	\begin{theorem}[\cite{MR1358272}, Thm. 3.4] \label{thm: long AMDS is NMDS}
		For $k\ge3$, if $m^1_t(k,q)>q+k$, then $t=1$.
	\end{theorem}

	In \cite{Tong2012}, Tong discusses the case $s=2$, establishing the following bound in the binary case.

	\begin{proposition}[\cite{Tong2012}, Cor. 2]  \label{prop: Tong binary}
		If $k\ge 3$ then $m^2(k,2)\le k+5$.
	\end{proposition}

	Faldum and Willems \cite{MR1432708} opened discussions to codes of arbitrary defect, providing bounds on $m^s(k,q)$. In particular, they provided the following result which subsumes part \ref{part: thm Dodunekov P1} of Theorem \ref{thm: Dodunekov 2.7}, and Theorem \ref{thm: long AMDS is NMDS}.

	\begin{lemma}[\cite{MR1432708}, Lem. 2] \label{lem: long bound} \leavevmode \vspace*{-2mm}
		\begin{enumerate}
			\item If $k\ge 2$, then $m^s(k,q)\le (s+1)(q+1)+k-2$. \label{part: lem long bound P1}
			\item If $k\ge 3$, and $m^s(k,q) = (s+1)(q+1)+k-2$ then $s\le q-1$. \label{part: lem long bound P2}
		\end{enumerate}
	\end{lemma}

	If a code achieves the bound in Lemma \ref{lem: long bound} then it is said to be \textit{length-maximal}. Faldum and Willems go on to show a code of sufficient length must be dual to an AMDS code.
	\begin{theorem}[\cite{MR1432708}, Thm. 7] \label{long AsMDS dual to AMDS}
		If $s\ge 1$ and $k\ge 2$, and $m^s_t(k,q)>s(q+1)+k-1$, then $t\le 1$.
	\end{theorem}

	\begin{corollary}[\cite{MR1432708}, Lem. 2]\label{cor: FW-k-bounds} Let $\G$ be an $[n,k,d]_q$ A$^s$MDS code with $k\ge 2$ and $n>s(q+1)+k-1$.
		\begin{enumerate}
			\item If $s=1$, then $k\le 2q$;
			\item If $s>1$, then $k\le q$, and consequently (Lemma \ref{lem: long bound}) $n\le (q+1)(s+2)-3$.
		\end{enumerate}
	\end{corollary}

	In the sequel we shall establish bounds on $m^{s}_t(k,q)$ which subsume and generalize the existing bounds. We are also interested in the existence question of length-maximal codes. Our work  leads us to conjecture that such codes do not exist for code dimension $k>4$.

	Let us briefly describe code shortening, a technique that shall be employed in the sequel.

	\subsection*{Shortening codes via quotient geometries}

	Code shortening is a common technique employed in coding theory (see e.g. constructions Y1--Y4, p. 592 in \cite{MR0465509}). We shall employ shortening in some of our proofs. Within the framework of projective systems, shortening can be performed by way of quotients, described as follows.\\
	Suppose $\G$ is an $[n,k,d]_q$ \AS code in $\Sigma=\text{PG}(k-1,q)$, let $\Lambda$ be an $\ell$-flat (i.e.\ an $\ell$-dimensional projective subspace) of codimension $r\ge 2$, and let $\alpha=|\Lambda \cap \G|$ (including multiplicities).  Let $\Sigma^*$ be the quotient geometry taken at $\Lambda$, so that in particular, $\Sigma^*\cong \text{PG}(r-1,q)$.  Each point $P$ of $\Sigma^*$ is the image of an $(\ell+1)$-flat, $\Omega_P$, where $\Lambda \subseteq \Omega_P$. If for each such point $P$, we assign the multiplicity $\mu(P) = |\Omega_P\cap \G|-\alpha$ then we obtain an $[n-\alpha, r, d^*]_q$ A$^{s^*}$MDS code $\G^*$, where $n-\alpha-d^*\le n-d-\alpha$, giving $d^*\ge d$, and
	\begin{equation}
		s^*=n-\alpha -r+1-d^* = n-\alpha +\ell-k+2-d^* \le n-\alpha +\ell-k+2-d =s-\alpha +\ell+1.
	\end{equation}

	Consequently, we have the following.

	\begin{proposition} \label{prop: quotient properties}
		Let $\G$ and $\G^*$ be as above.

		\begin{enumerate}
			\item If $\alpha = |\Lambda \cap \G|\ge \ell+1$, then $s^*\le s$, and $n \le m^{s^*}(r,q)+\alpha$.
			\item  $d^*=d$ if and only if $\Lambda$ is contained in at least one secant of $\G$, and $s^*=s$ if and only if $d=d^*$, and $\alpha = \ell+1$. \\ In particular, if $\Lambda$ is a point of $\G$ (i.e.\ $\ell=0$ and $\Lambda \in \mathcal{G}$), then $d=d^*$ and $s^*=s-\mu(\Lambda)+1$.
		\end{enumerate}
	\end{proposition}

	\begin{remark}\label{rem: shortening is algebraic}
		The quotient construction is the geometric form of classical code shortening, and admits a purely algebraic description. When $\Lambda$ is spanned by a set $T=\{i_1,\dots,i_{\ell+1}\}$ of (independent) coordinates of $\G$, the resulting code $\G^*$ is precisely the iterated shortening of $C$ at the coordinates in $T$, namely $C_T=\{u\in C\mid u_i=0 \text{ for all } i\in T\}$ punctured on $T$; for $\ell=0$ this recovers the familiar $C_i=\{u\in C\mid u_i=0\}$. More generally, projecting from an arbitrary flat $\Lambda$ amounts, at the level of a generator matrix $G$, to passing each column to its image in $\Sigma^*=\Sigma/\Lambda$. Upon choosing coordinates with $\Lambda=\langle e_1,\dots,e_{\ell+1}\rangle$, this is simply the deletion of the first $\ell+1$ rows of $G$. The only points of $\G$ mapping to the zero of the quotient are those lying in $\Lambda$. Once they are removed, $\G^*$ is non-degenerate, since (by assumption) $\G$ is not contained in a hyperplane of $\Sigma$.
	\end{remark}

	Before delving into the main results of this paper, we explore bounds that follow from the basic theory, as well as those that follow from the established bounds on planar arcs.

For ease of reference, Table~\ref{table: notation} collects the recurring notation. We follow the convention that Roman letters denote code- and arc-theoretic quantities, while Greek letters denote geometric objects (subspaces and flats) and point multiplicities; the one traditional exception is $H$, reserved for a hyperplane.

\begin{table}[h]
	\centering
	\caption{Summary of notation. Roman letters denote code/arc parameters;
		Greek letters denote geometric objects (subspaces, flats) and point
		multiplicities.}\label{table: notation}
	\begin{tabular}{|c|l|}
		\hline
		Symbol & Meaning \\ \hline\hline
		\multicolumn{2}{|c|}{\emph{Code and arc parameters (Roman)}}\\ \hline
		$q$ & order of the finite field $\F_q$ \\ \hline
		$\G$ ($C$) & a linear code, viewed as a projective system \\ \hline
		$n$ & length of $\G$; $n=|\G|$ (counting multiplicity) \\ \hline
		$k$ & dimension of $\G$; $k=\dim\Sigma+1$ \\ \hline
		$d$ & minimum distance of $\G$ \\ \hline
		$s$ & Singleton defect of $\G$; $s=S(\G)=n-k+1-d$ \\ \hline
		$k^\perp,d^\perp$ & dimension, minimum distance of $\G^\perp$; $k^\perp=n-k$ \\ \hline
		$t$ & Singleton defect of $\G^\perp$; $t=S(\G^\perp)$ \\ \hline
		A$^s$MDS, N$^s$MDS & code with $S(\G)=s$; resp.\ $S(\G)=S(\G^\perp)=s$ \\ \hline
		$m^s(k,q)$ & max.\ length of a non-degenerate $[n,k,d]_q$ A$^s$MDS code \\ \hline
		$m^s_t(k,q)$ & as above, with $\G^\perp$ additionally A$^t$MDS \\ \hline
		$m(k,q),\,m'(k,q)$ & $m^0(k,q)$ and $m^1_1(k,q)$ respectively \\ \hline
		$\kappa(s,q)$ & max.\ $k$ with a length-maximal A$^s$MDS code \\ \hline
		$(n,r)$-arc & $n$ points, $\le r$ per hyperplane, $=r$ on some hyperplane \\ \hline
		$\mu(P)$ & multiplicity of the point $P$ in $\G$ \\ \hline
		secant, tangent & hyperplane meeting $\G$ in $n-d$, resp.\ $n-d-1$, points \\ \hline\hline
		\multicolumn{2}{|c|}{\emph{Geometric objects (Greek; $H$ for hyperplane by convention)}}\\ \hline
		PG$(k,q)$ & projective space of dimension $k$ over $\F_q$ \\ \hline
		$m$-flat & projective subspace of (projective) dimension $m$ \\ \hline
		$\Sigma$ & ambient space PG$(k-1,q)$ containing $\G$ \\ \hline
		$\Sigma^*$ & quotient geometry (used in shortening) \\ \hline
		$H$ & a hyperplane of $\Sigma$ \\ \hline
		$\Lambda,\Omega,\Gamma$ & flats (of various dimensions) \\ \hline
		$\langle\mathcal{Q}\rangle$ & span of the point set $\mathcal{Q}$ \\ \hline
	\end{tabular}
\end{table}

	\section{Trivial codes and their bounds} \label{sec: trivial codes}
	We begin with 1-dimensional codes, where we temporarily widen the discussion to include degenerate codes. For any $n,d,q$ one may construct the trivial (possibly degenerate) $[n,1,d]_q$ \AS code by taking as $\G$ a multiset from  $\{0,1\}$ ($=$PG$(0,q) \cup\{0\})$, where $s=\mu(0)$, and $d=\mu(1)$. By definition, $\G$  will be MDS and non-degenerate if $d=n$, and degenerate with $s>0$ otherwise. The dual code $\G^\perp$ is an $[n,n-1,d^\perp]_q$ code, which is MDS ($d^\perp = 2$) if $d=n$, and is AMDS ($d^\perp = 1$) otherwise.

	In the 2-dimensional case, for arbitrary $n,d,q$, and $s=n-d-1$, one may clearly select a multiset of points from PG$(1,q)$ with the maximum multiplicity of any point being $n-d$.  As such, one may construct an $[n,2,d]_q$-\AS code if and only if $2+s\le n\le (s+1)(q+1)$. If $\G$ is such a code with $\G^\perp$ an $[n,n-2,d^\perp]_q$ code, then (Eqn. \ref{eqn: Proj Sys dual distance}) $d^\perp =3$ if $s=0$, and $d^\perp=2$ otherwise. For $k>2q$,  any $[k+2,2,d]_q$-code  necessarily satisfies $d^\perp =2 $ and is therefore dual to a $[k+2,k,2]_q$-AMDS code.
	Since a $[k+2,2,d]_q$ code can only be AMDS if $k+2\le 2(q+1)$, part \ref{part: thm Dodunekov P4} of Theorem  \ref{thm: Dodunekov 2.7} follows immediately, as does each of the following.

	\begin{lemma} \label{lem: trivial bounds} \leavevmode \vspace*{-2mm}
		\begin{enumerate}
			\item If $C$ is an $[n,2,d]_q$ \AS code, $s>0$, then $C^\perp$ is AMDS. \label{part 1 lem: trivial bounds}
			\item $m(1,q)$ is unbounded, and $m^s(1,q)=0$ for $s>0$; \label{part 2 lem: trivial bounds}
			\item $m^1(k,q)\ge k+1$, and $m^1(k,q)\ge k+2$ for $k>2q$. \label{part 3 lem: trivial bounds}
			\item $m'(k,q)\ge k+2$ for $2\le k \le 2q$. \label{part 4 lem: trivial bounds}
			\item $m^s(2,q)=(s+1)(q+1)$ \label{part 5 lem: trivial bounds}
		\end{enumerate}
	\end{lemma}

	\begin{remark}
		Note that part  \ref{part: thm Dodunekov P2} of Theorem \ref{thm: Dodunekov 2.7}   is subsumed by Lemma \ref{lem: trivial bounds}(\ref{part 5 lem: trivial bounds}).
	\end{remark}

	Let us denote by $e_i$, $1\le i \le k$, the projective point with homogeneous coordinates being the $i$'th standard basis vector in $\F_q^k$.  For fixed $k,q,s$, the code $\G$ comprising $ \{e_1,e_2,e_3,\ldots,e_k\}$, where  $\mu(e_1)=s+1$, and $\mu(e_i)=1$ for $2\le i \le k$, is an $[k+s,k,1]_q$ \AS code, with $d=1$, and $d^\perp=2$. We thus have the following.
	\begin{lemma}
		$m^s(k,q)\ge k+s$.
	\end{lemma}

	Codes of interest here are those which are non-degenerate and ``long''---in the sense of being (near) length-maximal. If $\G$ is an $[n,k,d]_q$ \AS code with $d^\perp=1$ then $\G$ is degenerate, and if $d=1$, then $n=k+s$, so $\G$ is far from length-maximal.   Thus, having discussed codes of dimension $2$ or less, in the sequel we shall generally limit discussion to cases  where $d,d^\perp>1$ and $k>2$.

	\section{Some established code bounds based on arcs}

	We now explore the relationship between $(n,r)$-arcs in projective planes and the bounds on the lengths of linear codes.

	\subsection{MDS Codes} \label{subsec: MDS Codes}

	In 1952,  Bush \cite{Bush1952} established that if $k\ge q$ then $m(k,q) = k+1$. For $k<q$ there is a long-standing conjecture regarding linear MDS codes: every linear $[n,k,n-k+1]_q$ MDS code  with $1 < k < q$ satisfies $n \le q + 1$, except when $q$ is even and $k = 3$ or $k = q - 1$ in which cases $n \le q + 2$. This conjecture is called the main conjecture on linear MDS codes. Though the main conjecture has been shown to hold in many cases, the topic is far from moribund. Below we list some of the cases where the main conjecture holds, but for a more complete summary see \cite{MR2061806} and \cite{Ball2012}.

	\begin{lemma}\label{lem: Bounds on MDS codes}
		Throughout, assume $q>k$.
		\begin{enumerate}
			\item $m(k,q)= q+1$ for $k=2,4, 5$.  \label{part 1: lem: Bounds on MDS codes}
			\item $m(3,q)= q+1$ if $q$ is odd, and  $m(3,q)= q+2$ if $q$ is even. \label{part 2: lem: Bounds on MDS codes}
			\item For $k\ge 4 $, $m(k,q)\le q+k-3$. 	\label{part 3: lem: Bounds on MDS codes}
			\item If $q=p^h$ and $4\le k \le p$, then $m(k,q)=q+1$.  \label{part 4: lem: Bounds on MDS codes}
			\item If $q=p^h$, $h>1$ and $ k \le 2p-2$, then $m(k,q)=q+1$.  \label{part 5: lem: Bounds on MDS codes}
			\item If $q=p^{2h}$,  and $ k \le \sqrt{q}-\sqrt{q}/p+2$, then $m(k,q)=q+1$. \label{part 6: lem: Bounds on MDS codes}
		\end{enumerate}
	\end{lemma}
	\begin{proof}
		Part 1 is clear for $k=2$. For $k=4,5$, the case $q$ odd is due to Segre \cite{MR0075608}, whereas the case $q$ even is due to Casse \cite{MR684151}, and  Gulati and Kounias \cite{Gulati1973}.  For part 2 see \cite{Bose1947} for $q$ odd, and for $q$ even see e.g. \cite{Hirschfeld1979}. For part 3, see \cite{MR0075608} for $q>4$ odd,  \cite{MR0253138} for $q>4$ even. For the remaining parts, see \cite{Ball2012a}, \cite{Ball2019}.
	\end{proof}

	\subsection{Bounds on 3-dimensional codes}

	The results of Barlotti (1956) \cite{MR0083141}, when paired with the later developments of Ball et al.\ \cite{MR1466573}, give bounds on $(n,r)$-arcs in PG$(2,q)$ which provide the first three items in the following lemma. The fourth item is due to the construction of Denniston \cite{MR0239991}.

	\begin{lemma}[\cite{MR0083141},\cite{MR1466573}, \cite{MR0239991}]  \label{lem: Barlotti 3d}  \leavevmode \vspace*{-2mm}
		\begin{enumerate}
			\item $m^s(3,q)\le (s+1)(q+1)+1$.
			\item  If  $0<s< q-2$, and $(s+2,q) \ne (2^e,2^h)$,  then $m^s(3,q)\le (s+1)(q+1)-1$.
			\item  If  $0<s< q-2$, $(s+2)|q$, and $m^s(3,q) \ge (s+1)(q+1)$, then $ m^s(3,q)= (s+1)(q+1)+1$.
			\item  If  $0<s\le q-2$, and $(s+2,q) = (2^e,2^h)$,  then $m^s(3,q)= (s+1)(q+1)+1$.
		\end{enumerate}
	\end{lemma}

	Further, we have the following.

	\begin{lemma} \label{lem: s-arcs in the plane}
		$m^s(3,q)\le (s+1)q+1$ in each of the following cases:
		\begin{enumerate}
			\item $q$ is prime, and $s+2\le (q+3)/2$ (\cite{MR1413900}) \label{part: 1, prime case  lem: s-arcs in the plane}
			\item $\gcd(s+2,q)=1$ and $s<\sqrt{2q}-1$  (\cite{MR2158768} ) \label{part: 2, lem: s-arcs in the plane}
			\item $q$ odd, $s+2$ divides $q$, $s<\frac14\sqrt{q}-2$ (\cite{MR2067604}) \label{part: 4, lem: s-arcs in the plane}
			\item $\gcd(s+2,q)=1$ and there is a line disjoint from the corresponding projective system. (\cite{MR1317133}) \label{part: 3, lem: s-arcs in the plane}
		\end{enumerate}
	\end{lemma}

	\section{Main Results}

	In this section, we present the primary findings, focusing on the bounds for the lengths of codes with non-zero Singleton defects. We first establish some elementary properties of $m^s(k,q)$.

	\begin{lemma}\label{lem: bounds on length of AsMDS} The following hold.
		\begin{enumerate}
			\item $m^{s}(k,q)<m^{s+1}(k,q)$.  \label{part: 1 bounds on length of AsMDS}
			\item For $k\ge 2$, $m^{s}(k,q)\le m^{s}(k-1,q)+1$.  \label{part: 2 bounds on length of AsMDS}
			\item For $\alpha<k$, $m^{s}(k,q)\le m^{s}(k-\alpha,q)+\alpha$. In particular, for $k\ge 3$, $m^{s}(k,q)\le m^{s}(3,q)+k-3$.   \label{part: 3 bounds on length of AsMDS}
			\item %$m^{s_1}(k,q)+m^{s_2}(k,q)\le m^{s_1+s_2+k-1}(k,q)$.\\
			If $s=k-1+b$, where $b=s_1+s_2$, $s_1,s_2\ge 0$, then $m^s(k,q)\ge m^{s_1}(k,q)+m^{s_2}(k,q).$ \label{part: 4 bounds on length of AsMDS}
			\item $m^s(3,q)\ge \left\{ \begin{array}{ll}
				\frac{s+2}{2}\cdot  m(3,q)	& \text{ if $s$ is even, } \\
				\frac{s-1}{2}\cdot m(3,q) +m^1(3,q)	&  \text{ if $s$ is odd.}
			\end{array}\right. $ \label{part: 5 bounds on length of AsMDS}
		\end{enumerate}
	\end{lemma}
	\begin{proof}
		Let $\G$ be an $[n,k,d]_q$ \AS code, let $\mathcal{H}$ be a secant of $\G$, and let $P\in \G\cap \mathcal{H}$.   If $n=m^{s}(k,q)$, then  $\G'=\G\cup\{P\}$ is an $[n+1,k,d]_q$ A$^{s+1}$MDS code (by the maximality of $n$), giving part \ref{part: 1 bounds on length of AsMDS}.
		For part \ref{part: 2 bounds on length of AsMDS},  suppose $P$ has multiplicity $m$. Taking the quotient geometry by $P$, $\G^*=\G\setminus \{P\}$ corresponds to an $[n-m,k-1,d]_q$ code in  PG$(k-2,q)$. By part \ref{part: 1 bounds on length of AsMDS}, $n-m\le m^{s-m+1}(k-1,q)\le m^s(k-1,q)-(m-1)$. Part \ref{part: 3 bounds on length of AsMDS} follows inductively from part \ref{part: 2 bounds on length of AsMDS}.\\
		For part \ref{part: 4 bounds on length of AsMDS}, observe that if $\G_1$ is an $[n_1,k,d_1]_q$ A$^{s_1}$MDS code, and $\G_2$ is an $[n_2,k,d_2]_q$ A$^{s_2}$MDS code then each hyperplane of PG$(k-1,q)$ holds at most $k+s_1-1$ points of $\G_1$, and at most $k+s_2-1$ points of $\G_2$. Taking $\G=\G_1\cup\G_2$ yields an $[n_1+n_2,k,d]_q$ \AS code where $s=k-1+s_1+s_2$.\\
		Part \ref{part: 5 bounds on length of AsMDS} is obtained by recursively applying  part \ref{part: 4 bounds on length of AsMDS}.
	\end{proof}

	\begin{remark}
		Note that part  \ref{part: thm Dodunekov P3} of Theorem \ref{thm: Dodunekov 2.7}   is subsumed by Lemma \ref{lem: bounds on length of AsMDS}(\ref{part: 3 bounds on length of AsMDS}).
	\end{remark}

	\begin{remark}\label{rem: degenerate can be ignored}
		Note that if $\G$ is a degenerate $[n,k,d]_q$ \AS code, having $\alpha>0$ all-zero coordinates, then deleting these coordinates from $\G$ results in a non-degenerate $[n-\alpha, k, d]_q$ A$^{s-\alpha}$MDS code with $n-\alpha \le m^{s-\alpha}(k,q) $. By part \ref{part: 1 bounds on length of AsMDS} of Lemma \ref{lem: bounds on length of AsMDS}, we see that $\G$ satisfies $n\le m^{s-\alpha}(k,q)+\alpha \le m^s(k,q)$. Consequently, in searching for long codes, it suffices to consider only those that are non-degenerate.
	\end{remark}

	The bounds in Lemmas \ref{lem: Barlotti 3d}  and \ref{lem: s-arcs in the plane} may be extended to higher dimensions as follows.
	\begin{corollary}\label{cor: from Barlotti} For $k\ge 3$, the following hold.
		\begin{enumerate}
			\item  (cf. Lemma \ref{lem: long bound} (\ref{part: lem long bound P1}))  $m^s(k,q)\le (s+1)(q+1)+k-2$. \label{part 1 cor: from Barlotti}
			\item  If  $0<s< q-2$, and $(s+2,q) \ne (2^e,2^h)$, then $m^s(k,q)\le (s+1)(q+1)+k-4$. \label{part 2 cor: from Barlotti}
			\item Under the conditions specified in parts 1--3 of Lemma \ref{lem: s-arcs in the plane}, $m^s(k,q)\le q(s+1)+k-2$. \label{part 4 cor: from Barlotti}
			\item  If $\gcd(s+2,q)=1$, $s\le q$, $\G$ is an $[n,k,d]_q$ \AS code, and some hyperplane $H$ satisfies $|H\cap \G|=k-3$, then $n\le q(s+1)+k-2$. \label{part 6 cor: from Barlotti}
		\end{enumerate}
	\end{corollary}

	\begin{remark}
		Note that Proposition \ref{prop: AMDS de Boer} is a particular case of part \ref{part 2 cor: from Barlotti} of Corollary \ref{cor: from Barlotti}.
	\end{remark}

	\begin{proof}
		Parts \ref{part 1 cor: from Barlotti}, \ref{part 2 cor: from Barlotti},  and \ref{part 4 cor: from Barlotti} follow immediately from the application of part \ref{part: 2 bounds on length of AsMDS} of Lemma \ref{lem: bounds on length of AsMDS} to Lemmas \ref{lem: Barlotti 3d}  and \ref{lem: s-arcs in the plane}.
		For part \ref{part 6 cor: from Barlotti}, let $H\cap \G=S$, and let $\Omega$ be a $(k-4)$-flat with $S\subseteq \Omega\subseteq H$. Taking the quotient through $\Omega$ yields an $[n-k+3, 3,d']_q$ A$^{s'}$MDS code $\G'$, where $d'\ge d$, $s'\le s$, and the line corresponding to $H$ is disjoint from $\G'$. If $s'=s$, then since $\gcd(s'+2,q)=\gcd(s+2,q)=1$, part \ref{part: 3, lem: s-arcs in the plane} of Lemma \ref{lem: s-arcs in the plane} applies to $\G'$, giving $n-k+3\le (s+1)q+1$, so $n\le q(s+1)+k-2$. If instead $s'<s$, then part 1 of Lemma \ref{lem: Barlotti 3d} gives $n-k+3\le (s'+1)(q+1)+1\le s(q+1)+1$, so that $n\le s(q+1)+k-2\le q(s+1)+k-2$, the last inequality holding as $s\le q$. In either case the result follows.
	\end{proof}

	\subsection{Projective codes}

	Recall that a code $\G$ is projective if $\mathcal{G}$ is a set, i.e.\ every point has multiplicity $1$. We now establish conditions under which codes are necessarily projective, and derive bounds on their lengths.

	Note that according to equation (\ref{eqn: Proj Sys dual distance}), a (non-degenerate) code of dimension $k>2$ is projective if and only if $d^\perp\ne 2$.

	\begin{theorem}\label{thm-Projective}
		Let $\G$ be an $[n,k,d]_q$ \AS code, $k\ge 2$.
		\begin{enumerate}
			\item If $s=0$ then $\G$ is projective.\label{part: 1 thm proj}
			\item If $k=2$, then $\G$ is projective if and only if $s=0$. \label{part: 2 thm proj}
			\item If $s>0$, $k>2$, and $n> m^{s-1}(k-1,q)+2$ then $\G$ is projective. \label{part: 3 thm proj}
			\item For $k>2$, if  $n> s(q+1)+k-1$  then $\G$ is projective. \label{part: 4 thm proj}
			\item If $s=1$, $k>q$, and $n> k+2$  then $\G$ is projective. \label{part: 5 thm proj}
		\end{enumerate}
	\end{theorem}
	\begin{proof}
		Part \ref{part: 1 thm proj} follows from the definition of MDS codes. In $\Sigma=$ PG$(1,q)$, points and hyperplanes coincide, so according to Definition \ref{defn: linear Code by PS},  an $[n,2,d]_q$ code is projective if and only if $n-d=1$, giving part \ref{part: 2 thm proj}.

		For the third part, let $\G=\{P_1,P_2,\ldots,P_n\}$ be an $[n,k,d]_q$ \AS code where, say $\mu(P_n)=m\ge2$. Taking the quotient at $P_n$, the projective system $\G'=\G\setminus \{P_n\}$ corresponds to an $[n-m, k-1,d]_q$ A$^{s'}$MDS code where $s'=s-m+1$, giving $n\le m^{s'}(k-1,q)+m\le m^{s-1}(k-1,q)+2$ (by Lemma \ref{lem: bounds on length of AsMDS}).

		For part \ref{part: 4 thm proj}, suppose $\G$ is an $[n,k,d]_q$-A$^s$MDS code, and $ \mu(P)=m\ge 2$. Through $P$, there exists a $(k-3)$-flat $\Gamma$ with $|\Gamma \cap \G|=x\ge k-1$.  Each of the $q+1$ hyperplanes through $\Gamma$ meets $\G$ in at most $k+s-1-x\le s$ points outside of $\Gamma$, giving the result.

		Part \ref{part: 5 thm proj} follows from part \ref{part: 3 thm proj} and the fact that $m(k-1,q)=k$ for $k-1\ge q$.
	\end{proof}

	\begin{remark}
		By part \ref{part: 3 thm proj}, an $[n,k,n-k]_q$ AMDS-code $\G$ is necessarily projective if $n>m(k-1,q)+2$.  With the current state of the Main Conjecture on linear MDS codes, in many cases $m(k-1,q)=q+1$, whence $\G$ is projective if $n>q+3$.
	\end{remark}

	We may also deduce the following bound on \AS codes that are not dual to an AMDS code.

	\begin{lemma} \label{lem: bounds based on d perp}
		If $k\ge 3$, and $t>1$ then
		\[
		m^s_t(k,q)\le m^{s-1}(t,q)+d^\perp = m^{s-1}(t,q)+k-t+1.
		\]
	\end{lemma}
	\begin{proof}
		Note that since $t>1$, $s>0$. Let $\G$ be an $[n,k,d]_q$ \AS code. If $d^\perp=2$ then $t=k-1$ and $\G$ is not projective, so the result follows from Theorem \ref{thm-Projective}.  Let $d^\perp \ge 3$, so that in particular, $\G$ is projective. There exist $d^\perp$ points of $\G$ incident with a common $(d^\perp-2)$-flat, $\lambda$, and any $d^\perp-1$ or fewer points of $\G$ are independent. Consider a $(d^\perp-3)$-flat, $\lambda'$  spanned by $d^\perp-2$ points of $\G\cap \lambda$. Taking the quotient by $\lambda'$ yields  an $[n-d^\perp+2,k-d^\perp+2,d^*]_q$ A$^{s^*}$MDS code, $\G^*$, with  $s^*\le s$ (Prop. \ref{prop: quotient properties}). Since the point in the quotient corresponding to $\lambda$ has multiplicity at least $2$, $\G^*$ is not projective, consequently (Theorem \ref{thm-Projective}),   $n-d^\perp+2\le m^{s^*-1}(k-d^\perp+1,q)+2\le m^{s-1}(k-d^\perp+1,q)+2$.
	\end{proof}

	From Lemma \ref{lem: bounds on length of AsMDS},  if $t\ge 3$ and $s>0$ then $m^{s-1}(t,q)\le m^{s-1}(3,q)+t-3$. This bound may in turn be leveraged by bounds on $(n,s+1)$-arcs in the plane, such as those in Lemmas \ref{lem: Barlotti 3d}, and \ref{lem: s-arcs in the plane}.
	Specializing to the case $s=1$: if $t\ge 2$ then $m(t,q)\le m(2,q)+t-2=q+t-1$. From part 1 of Lemma \ref{lem: trivial bounds}, the dual of any 2-dimensional non-MDS code is AMDS. Furthermore, if $t\ge q$, then as observed in Section \ref{subsec: MDS Codes}, $m(t,q)=t+1$.  We thus have the following corollary, the second part of which appears as Theorem 3.4 of \cite{MR1358272}.

	\begin{corollary} \label{cor: bound s = 1 t>1}  \leavevmode \vspace*{-2mm}
		\begin{enumerate}
			\item If $t\ge 3$ and $s>0$, then $m^s_t(k,q)\le m^{s-1}(3,q)+k-2 $
			\item If $k,t\ge 2$, then  $m^1_t(k,q)\le q+k$.
			\item If $t\ge q$ then  $m^1_t(k,q)\le k+2$.
		\end{enumerate}
	\end{corollary}

	\begin{remark}\label{rem: MDS Conj and AMDS t>1}
		With reference to Lemma \ref{lem: Bounds on MDS codes}, it is often the case that $m(t,q)=q+1$.  In such cases, the above corollary provides $m^1_t(k,q)\le q+k+2-t$ when $t>1$.
	\end{remark}

	\subsection{Bounds on \AS codes}

	It is noteworthy that $S(C^\perp) =t > 1$ if and only if $d^\perp < k$. Thus, for $k = 3$, all projective codes are either MDS or dual to an AMDS code. For dimensions $k > 3$, there exist projective codes that are neither MDS nor dual to an AMDS code. However, such codes are shown to necessarily be short, indicating that for dimensions three or greater, the longest codes are typically dual to AMDS codes.
	Ball and Hirschfeld observed in \cite{MR2158768} that, in most cases, no examples exist of  $(n, r)$-arcs in $\text{PG}(2, q)$ with a large $n/q$ ratio, specifically $n/q > r - 2$. The following demonstrates that for dimensions $k > 3$, the longest codes that are neither AMDS nor dual to an AMDS code generally correspond (by shortening) to $(n, r)$-arcs in $\text{PG}(2, q)$ with $n \leq (r - 2)q + r$.

	\begin{remark}
		Recall that an $[n,k,d]_q$ \AS code $\G$  is degenerate if and only if $d^\perp=1$, equivalently $n=k^\perp+t$.  Likewise $\G^\perp $ is an $[n,k^\perp,d^\perp]_q$ A$^t$MDS code, and is degenerate if and only if $d=1$, in which case $n=k+s$. Thus, in our search for long non-degenerate codes we generally dismiss the case $d=1$, or equivalently, consider only $s<n-k$.
	\end{remark}

	\begin{theorem}\label{thm: ub ASMDS}
		Let $\G$ be an  $[n,k,d]_q$ \AS code, with  $\G^\perp$ an $[n,k^\perp,d^\perp]_q$ A$^t$MDS code, $d>1$, and $s, t\ge 1$.
		\begin{enumerate}
			\item  \begin{enumerate}
				\item If $t>1$ then $n\le s(q+1)+k-1$. \label{part: thm: ub ASMDS 1a }
				\item If $s>1$ then $k\le t(q+1)-1$. \label{part: thm: ub ASMDS 1b }
			\end{enumerate} \label{part: 1 thm ub ASMDS}
			\item If $t=1$ and $\G$ has a codeword of weight $d+s+1-\alpha$ for some $0\le \alpha\le s+1$ then $n\le q(s+1)+k-2+\alpha$. \label{part: 2 thm ub ASMDS}
			\item If $s=1$ then $n\le (t+1)(q+1)+k^\perp -2$, or equivalently, $k\le (t+1)(q+1)-2$. \label{part: 3 thm ub ASMDS}
			\item If $s=t=1$, $q>3$, and $k\ge 3,$  then $n\le 2q+k-2$, and if $n>k+2$ then $k\le 2q-2$. \label{part: 4 thm ub ASMDS}
			\item If $t=1$ then $k+s-1\le m(k-1,q)$, equivalently $n\le m(k-1,q) +k^\perp -s+1$.\label{part: 5 thm ub ASMDS}
			\item If $s=1$ then $k^\perp+t-1\le m(k^\perp-1,q)$, so $n\le m(k^\perp-1,q)+k-t+1$. \label{part: 6 thm ub ASMDS}
			\item If $s=t=1$ then $k\le m(k-1,q)$; $k^\perp \le m(k^\perp-1,q)$. \label{part: 7 thm ub ASMDS}
			\item $k\le (t+1)(q+1)-2$. \label{part: 8 thm ub ASMDS}
		\end{enumerate}
	\end{theorem}

	\begin{remark}
		Part \ref{part: thm: ub ASMDS 1a }: cf.\ Theorem \ref{long AsMDS dual to AMDS}; in the case $t=2$, cf.\ Corollary 1 in \cite{Tong2012};\\
		Corollary \ref {cor: FW-k-bounds} follows from parts \ref{part: 1 thm ub ASMDS} and \ref{part: 3 thm ub ASMDS};\\
		Part \ref{part: 6 thm ub ASMDS}: cf.\ Theorem 11 in \cite{MR1409442}.
	\end{remark}
	\begin{proof}
		Before proceeding with the proof, we first note that $d^\perp = k+1-t$, and from equation (\ref{eqn: Proj Sys dual distance}) it follows that  $k+1-t$ points of $\G$ are incident with a common $(k-1-t)$-flat, $\Omega$,  and any $k-t$ points of $\G$ are  independent.
		Let $H$ be a secant of $\G$, so $H$ is a hyperplane of $PG(k-1,q)$ with $|H\cap \G|= k-1+s$, and no hyperplane meets $\G$ in as many as $k+s$ points.
		Part \ref{part: 1 thm ub ASMDS}: If $t>1$ then $d^\perp\le k-1$, so $k\ge 3$, so Lemma \ref{lem: bounds based on d perp} (which gives $m^s_t(k,q)\le m^{s-1}(t,q)+k-t+1$) and part \ref{part: 3 bounds on length of AsMDS} of Lemma \ref{lem: bounds on length of AsMDS} (which gives $m^{s-1}(t,q)\le m^{s-1}(2,q)+t-2$) give
		\[
		m^s_t(k,q)\le m^{s-1}(2,q)+k -1 = s(q+1)+k-1.
		\]
		Similarly, if $s>1$ then by working with $\G^\perp$ and the fact that $d>1$ provides $k^\perp\ge 3$,  we obtain $n\le t(q+1)+k^\perp -1$, so
		\begin{equation*} \label{eqn: max length ASMDS t}
			k\le t(q+1)-1.
		\end{equation*}
		Part \ref{part: 2 thm ub ASMDS}: If $t=1$ then $d^\perp=k$, so any $ (k-1)$-subset of $\G$ is  independent. $\G$ has a codeword of weight $d+s+1-\alpha$ for some $0\le \alpha\le s+1$ if and only if there exists some hyperplane $\Pi$ meeting $\G$ in precisely $k-2+\alpha$ points. Let $\Omega$ be the span of $ k-2 $ of these points.  Each hyperplane through $\Omega$ meets $\G$ in at most $s+1$ further points, giving
		\begin{equation} \label{eqn: bound on wt spectrum}
			n-k+2\le (s+1)(q) + \alpha.
		\end{equation}

		Part \ref{part: 3 thm ub ASMDS} follows from Part \ref{part: 2 thm ub ASMDS}, and the existence of a word of weight $d^\perp$ in $\G^\perp$.
		Part \ref{part: 4 thm ub ASMDS}: If $t=s=1$ then part \ref{part 2 cor: from Barlotti} of Corollary \ref{cor: from Barlotti} gives $n\le 2q+k-2$, and $n> k+2$ gives $k^\perp \ge 3$, whence $n\le 2q+k^\perp-2$.
		Part \ref{part: 5 thm ub ASMDS}: If $t=1$  then $H\cap \G$ is a set of  $k+s-1$ points such that every $(k-1)$-subset is  independent. In other words $H\cap \G$ is a $(k+s-1, k-2)$-arc in PG$(k-2,q)$.   Reasoning for Part \ref{part: 6 thm ub ASMDS}  is entirely similar, and applied to $\G^\perp$. Part \ref{part: 7 thm ub ASMDS} follows from parts \ref{part: 5 thm ub ASMDS} and \ref{part: 6 thm ub ASMDS}.
		Part \ref{part: 8 thm ub ASMDS} follows from parts \ref{part: 1 thm ub ASMDS} and \ref{part: 3 thm ub ASMDS}.
	\end{proof}

	Note that codes that achieve equality in part \ref{part 1 cor: from Barlotti} of Corollary \ref{cor: from Barlotti} have no codewords of weight $d+1$, whereas codes dual to those meeting the bound in part \ref{part: 8 thm ub ASMDS} of Theorem \ref{thm: ub ASMDS} have no codewords of weight $d^\perp +1$.

	\begin{corollary}\label{cor: AMDS k bound}  Let $\G$ be an $[n,k,d]_q$ \AS code, with  $\G^\perp$ an $[n,k^\perp,d^\perp]_q$ A$^t$MDS code,  $s, t\ge 1$.
		\begin{enumerate}
			\item If $n>s(q+1)+k-1$, then $s\le m(k-1,q)-k+1$. \label{part 1: cor: AMDS k bound}
			\item If $k> q$ and $s>1$, then  $n\le s(q+1)+k-1$. \label{part 2: cor: AMDS k bound}
			\item If $s>1$ and $t=1$ then $n\le q+k^\perp$, so $k\le q$. \label{part 3: cor: AMDS k bound}
			\item If $t>1$ and $s=1$ then $n\le q+k$, so $k^\perp \le q$. \label{part 4: cor: AMDS k bound}
		\end{enumerate}
	\end{corollary}
	\begin{proof}
		Part \ref{part 1: cor: AMDS k bound} follows from parts \ref{part: 1 thm ub ASMDS} and \ref{part: 5 thm ub ASMDS} of Theorem \ref{thm: ub ASMDS}. Part \ref{part 2: cor: AMDS k bound} clearly holds if $d=1$, and if $d>1$ then the result  follows from part 1 of Theorem  \ref{thm: ub ASMDS}. Parts  \ref{part 3: cor: AMDS k bound} and \ref{part 4: cor: AMDS k bound} are particular cases of part \ref{part: 1 thm ub ASMDS} of Theorem  \ref{thm: ub ASMDS}.
	\end{proof}

	Let $C$ be an \AS code, where $k\ge3$, $q>3$, and suppose $C^\perp$ is AMDS. By part \ref{part: 4 thm ub ASMDS} of Theorem \ref{thm: ub ASMDS}, if $s=1$ then $k\le 2q-2$, and if $s>1,$ then by part \ref{part: 1 thm ub ASMDS}, $k\le q$.  Thus, by part \ref{part 1 cor: from Barlotti} of Corollary \ref{cor: from Barlotti} we obtain the following.
	\begin{corollary} \label{cor: AMDS no k}
		If  $k\ge 3$, and $q>3$, then \[m^s_1(k,q)\le \left\{\begin{array}{ll}
			(s+3)(q+1)-6, & \text{ if } s=1\\
			(s+2)(q+1)-3, & \text{ if } s>1.
		\end{array}\right.\]
	\end{corollary}

	\begin{lemma} \label{lem: fat k-3 flat short code}
		If $\G$ is an $[n,k,d]_q$ \AS code and a $(k-3)$-flat $\Lambda$ meets $\G$ in $k-2+a$ points, $a\ge 0$, then
		\[
		n\le (s+1-a)(q+1)+ k-2+a
		\]
	\end{lemma}
	\begin{proof}
		Each hyperplane meets $\G$ in at most $n-d=k+s-1$ points of $\G$. Thus, by considering the pencil of hyperplanes through $\Lambda$ we obtain
		\[
		n\le k-2+a+(q+1)\bigl(k+s-1-(k-2+a)\bigr)
		\]
	\end{proof}
	As an immediate consequence, we have the following bound for non-projective codes.
	\begin{corollary}
		If $\G$ is an $[n,k,d]_q$ \AS  code where $\mu(P_i)=1+x_i$, $1\le i \le k-2$, then
		\[
		n\le (s+1-a)(q+1)+ k-2+a,
		\]
		where $a=\sum_{i=1}^{k-2}x_i$.
	\end{corollary}

	The next several corollaries of Lemma \ref{lem: fat k-3 flat short code} highlight that both codes of large minimum distance, and those of large defect must be short.

	\begin{corollary}\label{cor: big s short code}
		If $s\ge a (q+1)-1$, and $k\ge 3$, then  \[m^s(k,q)\le (s+1-a)(q+1)+k-2+a.\]
		In particular, if $s\ge q$, then  $m^s(k,q)\le s(q+1)+k-1$.
	\end{corollary}

	\begin{remark}
		Corollary \ref{cor: big s short code} immediately implies both Proposition \ref{prop: Tong binary} and part \ref{part: lem long bound P2} of Lemma \ref{lem: long bound}.
	\end{remark}

	\begin{proof}
		We shall deal with the cases $k=3$ and $k>3$ separately. Let $\G$ be an $[n,k,d]_q$ \AS code, $k>3$ and let $H$ be a secant of $\G$. If there exists a $(k-3)$-flat meeting $\G$ in at least $k-2+a$ points, then we are done (Lemma \ref{lem: fat k-3 flat short code}), so assume otherwise.  Let $\Omega\subseteq H $ be an $(k-4)$-flat with $|\Omega \cap \G| = k-3+x$. The assumption implies that $x<a$. By considering the respective intersections of $\G$ with the pencil of hyperplanes of $H$ through $\Omega$, we obtain
		\begin{align*}
			k+s-1& \le k-3+x+(q+1)\bigl(k-3-a-(k-3-x)\bigr) \\
			\Rightarrow\;\;\; & s \le a(q+1)-qx-2<a(q+1)-1
		\end{align*}
		giving a contradiction.
		If $k=3$ then a secant of $\G$ is a line $\ell$ meeting $\G$ in $s+2 \ge a(q+1)+1$ points. It follows that some point $P\in \ell$ has multiplicity $\mu(P)\ge a+1$, so $n\le a+1+ (q+1)(s+2-a+1)$.
	\end{proof}

	We may improve upon the bound in Corollary \ref{cor: big s short code} when $a=1$, provided $k$ and $q$ are such that $m(k-1,q)=q+1$ (as per the Main Conjecture on MDS Codes).
	\begin{corollary}\label{cor: MC big s short code}
		Let $k$ and $q$ be such that $m(k-1,q)=q+1$. If $s>q+2-k$, then $m^s(k,q)\le s(q+1)+k-1$.
		In particular, if $q$ is prime with $3\le k \le q$ then $m^s(k,q)\le s(q+1)+k-1$.
	\end{corollary}
	\begin{proof}
		Let $k$ and $q$ be as required, and let $\G$ be an $[n,k,d]_q$ \AS code. If $n>s(q+1)+k-1 $, then by part \ref{part: 1 thm ub ASMDS} of Theorem \ref{thm: ub ASMDS}, $t=1$, and by part \ref{part: 5 thm ub ASMDS}, we obtain
		\[
		s\le m(k-1,q)-k+1=q+2-k.
		\]
	\end{proof}

	\begin{corollary}\label{cor: big d means short}
		If $k\ge 3$, and $d>a (q^2+q)-q$, $a\ge 0$ then an $[n,k,d]_q$ \AS code satisfies
		\[n\le (s+1-a)(q+1)+k-2+a.\]
		In particular, if $d>q^2$, then $n\le s(q+1)+k-1$.
	\end{corollary}
	\begin{proof}
		For $k\ge 3$ let $\G$ be an $[n,k,d]_q$ \AS code with $d>a (q^2+q)-q$.  Let $H$ be a secant of $\G$, and let $S\subseteq H $ be a $(k-3)$-flat with $|S\cap \G|\ge k-2$.  Through $S$ there are $q$ hyperplanes other than $H$. Since $d>a (q^2+q)-q$, there exists a hyperplane $H'$ through $S$ with $|H'\cap \G|\ge k-2 + a q +a $. It follows that
		$n-d=k+s-1\ge k-2 +a q+a $, whence $s\ge a q+a-1 $, and Corollary \ref{cor: big s short code} gives the result.
	\end{proof}

	We may improve slightly if $t>1$.

	\begin{corollary}\label{cor: big d means short t>1}
		Let $\G$ be an $[n,k,d]_q$ \AS code, and $\G^\perp$ an $[n,k^\perp,d^\perp]_q$ A$^t$MDS code, $k\ge 3$. If $t>1$, and $d>a (q^2+q)-2q$, $a\ge 0$ then
		\[n\le (s+1-a)(q+1)+k-2+a.\]
	\end{corollary}
	\begin{proof}
		Since $t>1$ there exists a $(k-3)$-flat $\Omega$ with $|\Omega\cap \G|\ge k-1$. Let $H$ be a hyperplane through $\Omega$. By definition, $|H\cap \G|\le n-d$,  so (as in the previous proof) there exists a hyperplane $H'$ through $\Omega$ with $|H'\cap \G|\ge k-1 + a q +a -1$. It follows that
		$n-d=k+s-1\ge k-1+a q+a -1$, whence $s\ge a q+a-1$, and  Corollary \ref{cor: big s short code} gives the result.
	\end{proof}

	In \cite{MR2158768}, it is shown (for $k=3$) that if $d\le q^2$, then an $[n,k,d]_q$ \AS code $\G$ meets the Griesmer bound if and only if $n> s(q+1)+k-2$. The proof holds for $k\ge 3$. Thus, with Corollary \ref{cor: big d means short}, Theorem \ref{thm-Projective}, and part \ref{part: 1 thm ub ASMDS} of Theorem \ref{thm: ub ASMDS} we obtain the following.

	\begin{corollary}\label{cor: long implies Griesmer}
		For $k\ge 3$ let $\G$ be an $[n,k,d]_q$ \AS code with $\G^\perp$ an A$^t$MDS code. If $n> s(q+1)+k-2$ then   $\G$ meets the Griesmer bound, $\G$ is projective, and $t=1$.
	\end{corollary}

	\begin{corollary}\label{cor: big k, k+s} The following hold.
		\begin{enumerate}
			\item If $s\in \{0,1\}$ and $k\ge (t+1)(q+1)-1$, then $m^s_t(k,q) = k+1$. \label{part 1: cor: big k, k+s}
			\item If $s>1$ and $k\ge t(q+1)$, then  $m^s_t(k,q) = k+s$. \label{part 2: cor: big k, k+s}
		\end{enumerate}
	\end{corollary}
	\begin{proof}
		For the case $s=0$, see Bush \cite{Bush1952}. For the case $s=1$,  part  \ref{part: 3 thm ub ASMDS} of Theorem \ref{thm: ub ASMDS} gives $d=1$, so $n=k+s$. The case $s>1$ follows similarly from part \ref{part: 1 thm ub ASMDS} of Theorem \ref{thm: ub ASMDS}.
	\end{proof}

	The largest of the upper bounds on the length of linear $[n,k,d]_q$ codes presented here applies only to cases where $t=1$:
	\begin{equation} \label{eqn: largest bound}
		n\le (s+1)(q+1)+k-2.
	\end{equation}
	We are interested in determining existence conditions for codes meeting this bound. From Corollary \ref{cor: from Barlotti}, and Corollary \ref{cor: big s short code} any code of dimension $k\ge 4$ meeting the bound (\ref{eqn: largest bound}) satisfies either \begin{enumerate}
		\item[(a)] $(s+2)|q$,  $q$ even, and $2\le s \le q-4$, or
		\item[(b)] $q-2 \le s\le q-1$.
	\end{enumerate}
	There are clearly codes of dimension $k=2$ meeting (\ref{eqn: largest bound}) for each $s\ge 0$. When $k=3$ there are optimal codes of type (b). For example, removing a single line from the plane $PG(2,q)$ yields a $[q^2,3,q^2-q]_q$ code $C$ with $S(C)=q-2$, whereas the entire plane yields an $[q^2+q+1, 3, q^2]_q$ code $C'$, with $S(C')=q-1$. For each $q>2$, there exists a $(q^2+1)$-cap in PG$(3,q)$ with maximal plane intersection $q+1$, and there exists a  cap of size $8$ in PG$(3,2)$  (see e.g. \cite{MR840877}). Thus for each $q\ge 3$ there is a $[q^2+1,4,d]_q$ code $C$ with $S(C)=q-2$, and $S(C^\perp)=1$, and in the binary case there is an $[8,4,4]_2$ NMDS code ($s=q-1$). For dimensions $k\ge 5$ we are able to show that no codes of type (b) exist for $q>3$.
	\begin{corollary} \label{cor: q-1 and q-2} \phantom{ }
		\begin{enumerate}
			\item If $q>2$, and $2\le k \le 3$, or if $q=2$ and $2\le k \le 4$ then $m^{q-1}(k,q)=(s+1)(q+1)+k-2 = q^2+q+k-2$. \label{part: 1 cor: q-1 and q-2}
			\item If $q>2$, and $k\ge 4$,  then $m^{q-1}(k,q)\le s(q+1)+k-2 = q^2+k-3$. \label{part: 2 cor: q-1 and q-2}
			\item If $q=2$, and $k\ge 5$, then $m^{q-1}(k,q) = k+2$. \label{part: 3 cor: q-1 and q-2}
			\item If $2\le k\le 4$ and $q>2$  then $m^{q-2}(k,q)=(s+1)(q+1)+k-2 = q^2+k-3$. \label{part: 4 cor: q-1 and q-2}
			\item If $k\ge 5$ and $q>3$,  then $m^{q-2}(k,q)\le s(q+1)+k-1$. \label{part: 5 cor: q-1 and q-2}
		\end{enumerate}
	\end{corollary}
	\begin{proof}
		Parts \ref{part: 1 cor: q-1 and q-2}, and \ref{part: 4 cor: q-1 and q-2} follow from the preceding discussion.
		For the remaining parts, let $\G$ be an $[n,k,d]_q$ \AS	code with $ \G^\perp $ an A$^t$MDS code, and $n\ge s(q+1)+k$. Note that if $d=1$ then $n=k+s$, so $d>1$. By part \ref{part: 1 thm ub ASMDS} of Theorem  \ref{thm: ub ASMDS}, $t=1$, so if $k\ge 4$ then $\G$ is an $n$-cap in PG$(k-1,q)$.
		For part \ref{part: 2 cor: q-1 and q-2},  $s=q-1$, and $q>2$. Taking $k=4$ and $n=q^2+k-2$ yields a $(q^2+2)$-cap in PG$(3,q)$, contrary to the result of Bose \cite{Bose1947} for $q$ odd, and that of Qvist \cite{Qvist1952} for $q$ even.\\
		Part \ref{part: 3 cor: q-1 and q-2}: If $k\ge 5$ then (part \ref{part 3 lem: trivial bounds} of Lemma \ref{lem: trivial bounds}), $m^1(k,q)\ge k+2$. The result then follows from part \ref{part: thm: ub ASMDS 1a } of Theorem \ref{thm: ub ASMDS} (which gives $n\le s(q+1)+k-1$ when $t>1$), and part \ref{part: thm Dodunekov P4} of Theorem \ref{thm: Dodunekov 2.7} (which gives $m'(k,q)=k+1$ for $k>2q$).\\
		For part \ref{part: 5 cor: q-1 and q-2},  $s=q-2$, $q>2$. Consider $k=5$.  By part \ref{part: 5 thm ub ASMDS} of Theorem  \ref{thm: ub ASMDS} (which gives $k+s-1\le m(k-1,q)$), $m(4,q)\ge q+2$. The result of Bush \cite{Bush1952} (see introduction to Section \ref{subsec: MDS Codes}) gives $m(4,4)=5$, so consider $q>4$. Part \ref{part 3: lem: Bounds on MDS codes} of Lemma \ref{lem: Bounds on MDS codes} gives the contradiction $ m(4,q)\le q+1$  The bound for $k=5$ being established, the cases $k>5$ follow inductively from part \ref{part: 2 bounds on length of AsMDS} of Lemma \ref{lem: bounds on length of AsMDS}.
	\end{proof}

	\begin{lemma}\label{lem: full length comb}
		Let $k\ge 3$. If $\G$ is an $[n,k,d]_q$ \AS code with $n=(s+1)(q+1)+k-2$, then for each $0\le j\le k-2$,    $\gamma_j = \frac{{n-j\choose k-1-j}}{{n-d-j \choose k-1-j}}$ is an integer.
	\end{lemma}
	\begin{proof}
		Suppose $\G$ is an $[n,k,d]_q$ \AS code with $n= (s+1)(q+1)+k-2$. From Theorem \ref{thm: ub ASMDS} $S(\G^\perp)=1$, so $d^\perp=k$. Consequently, any subset of $k-1$ or fewer points of $\G$ are independent. In particular, if $S \subseteq \G$ with $|S|=k-2$,  then $\langle S \rangle$ is a $(k-3)$-flat with $|\langle S \rangle \cap \G| = k-2$.  Simple counting shows that each member of the pencil of $q+1$ hyperplanes through $\langle S \rangle$  meets $\G$ in exactly $s+k-1$ points (i.e. each is a secant of $\G$). Denote by $\gamma_0$ the number of secants. Counting  set, hyperplane pairs $(A,\mathcal{S})$ where $A\subseteq \G$, $|A|=k-1$, $A \subseteq \mathcal{S}$, and $\mathcal{S}$ is a secant of $\G$ we obtain
		\begin{equation}
			{n\choose k-1} = \gamma_0 \cdot {k+s-1 \choose k-1},
		\end{equation}
		establishing the result for $j=0$.  Let $H$ be a secant of $\G$, and let  $B\subseteq \G\cap H$ with $|B|=j$, $1\le j\le k-2$. Let $\G^*$ be the points in the  quotient by $\langle B \rangle$ corresponding to $\G\setminus B$. Since $\G$ is necessarily projective, and $S(\G^\perp) = 1$, $\G^*$ is a non-degenerate, projective $[n-j, k-j,d]_q$-code. Thus the result follows inductively.
	\end{proof}

	\begin{lemma}\label{lem: alpha beta}
		Let $k\ge 3$. If $\G$ is an $[n,k,d]_q$ \AS code with $n=(s+1)(q+1)+k-3$, then for each $0\le j\le k-3$,  $\alpha_j$, and $\beta_j$ are integers, where $\alpha_j = \frac{{n-j\choose k-2-j}}{{k+s-2-j \choose k-2-j}}$ and  $\beta_j = \alpha_j \cdot  \frac{q(s+1)}{k+s-1-j}$.
	\end{lemma}
	\begin{proof}
		Suppose $\G$ is an $[n,k,d]_q$ \AS code with $n= (s+1)(q+1)+k-3$. As in the previous proof, $S(\G^\perp)=1$, and if  $S\subseteq \G$ with $|S|=k-2$, then  $\langle S \rangle$ is a $(k-3)$-flat with $\langle S \rangle \cap \G = k-2$.  Simple counting shows that the pencil of $q+1$ hyperplanes through $\langle S \rangle$ holds exactly one member meeting $\G$ in $s+k-2$ points (let us call such a hyperplane a tangent of $\G$), the remaining $q$ members are secants of $\G$. Denote by $\alpha_0$ the number of tangents of $\G$, and by $\beta_0$ the number of secants. Counting  $(k-3)$-flat, hyperplane incident pairs $(\mathcal{A},\mathcal{T})$ where $\mathcal{A}\cap \G = k-2$, and $\mathcal{T}$ is a tangent of $\G$ we obtain
		\begin{equation}
			{n\choose k-2}\cdot 1 = \alpha_0 \cdot {k+s-2 \choose k-2}.
		\end{equation}
		Similarly, counting  $(k-3)$-flat, hyperplane incident pairs $(\mathcal{B},\mathcal{S})$ where $\mathcal{B}\cap \G = k-2$, and $\mathcal{S}$ is a secant of $\G$ we obtain
		\begin{equation}
			\beta_0 = \frac{{n\choose k-2}\cdot q}{{k+s-1 \choose k-2}} \left(= \alpha_0 \cdot \frac{q(s+1)}{k+s-1}\right).
		\end{equation}
		As in the previous result, an inductive argument completes the proof.
	\end{proof}

	\begin{corollary}\label{cor: prime divisors}
		Let $\G$ be an $[n,k,d]_q$ \AS code.
		\begin{enumerate}
			\item If there exists a prime $p$ with $\prod_{i=1}^{s} (q(s+1)+i) \equiv 0 \mod p$ where $s<p<k+s$, then $n\le (s+1)(q+1)+k-3$. \label{part: 1 cor: prime divisors}
			\item If there exists a prime $p$ with $\prod_{i=1}^{s} (q(s+1)+i) \equiv 0 \mod p$ where $s<p<k+s-1$, then $n\le (s+1)(q+1)+k-4$. \label{part: 2 cor: prime divisors}
			\item If there exists a prime $p$ with $\prod_{i=0}^{s} (q(s+1)+i) \equiv 0 \mod p$, and gcd$(p,q)=1$ where $s+1<p<k+s-1$, then $n\le (s+1)(q+1)+k-4$. \label{part: 3 cor: prime divisors}
		\end{enumerate}
	\end{corollary}
	\begin{proof}
		For part 1, suppose $\G$ is an $[n,k,d]_q$ \AS code with $n=(s+1)(q+1)+k-2$, and let $p$ be a prime with $\prod_{i=1}^{s} (d+i) \equiv 0 \mod p$, $s<p<k+s$. By Lemma \ref{lem: full length comb}, for any $0\le j \le k-1$
		\begin{equation}
			\gamma_j = \frac{{n-j\choose k-1-j}}{{n-d-j \choose k-1-j}} = \frac{(n-j)!\cdot s!}{(d+s)!\cdot (n-j-d)!} =\frac{s!{n-j \choose d}}{(d+1)(d+2)\cdots (d+s)}
		\end{equation}
		is an integer.  Taking $j=n-d-p$, we obtain
		\begin{equation}
			\gamma_j  =\frac{s!{d+p \choose d}}{(d+1)(d+2)\cdots (d+s)} =\frac{s!(d+s+1)(d+s+2)\cdots (d+p)}{p!}.
		\end{equation}
		However, since $p\mid d+a$ for some $1\le a\le s$, $p$ does not divide $d+x$ for $s+1\le x <a+p$, whence $\gamma_j$ is not an integer.
		With reference to Lemma \ref{lem: alpha beta}, the proofs of parts 2 and 3 are entirely similar after observing that
		\[
		\alpha_j = \frac{s!{n-j \choose d+1}}{(d+2)(d+3)\cdots (d+s+1)}, \text{ and } \beta_j = \frac{q\cdot (s+1)!{n-j \choose d}}{(d+1)(d+2)\cdots (d+s+1)}.
		\]
	\end{proof}

	With reference to Lemma \ref{lem: full length comb}, part  \ref{part: 1 cor: prime divisors} of Corollary \ref{cor: prime divisors} gives the following.

	\begin{corollary}\label{cor: cor prime divisors}
		If $0<s< q-2$,  $(s+2)|q$, and there exists a prime $p$ with $\prod_{i=1}^{s} (q(s+1)+i) \equiv 0 \mod p$ where $s<p<k+s$, then
		\[
		m^s(k,q) \le (s+1)(q+1)+k-4.
		\]
	\end{corollary}

	\subsection*{Bounds on $\kappa (s,q)$ }
	For fixed $s,q$, let us denote by $\kappa (s,q)$ the maximum $k$ such that $m^s(k,q)=(s+1)(q+1)+k-2$. Table \ref{table: summary of kappa values} serves to set some of our results in context.

	\begin{table}[h]
		\centering
		\caption{Bounds on  $\kappa (s,q)$}\label{table: summary of kappa values}%
		\begin{tabular}{|c|c|c|}
			\hline
			{ Conditions }	& $\kappa(s,q)$  &  { Ref.}  \\ \hline
			$s\ge 0$	& $\ge 2$ & { Lem. \ref{lem: trivial bounds} } \\  \hline
			$s= 0,\, q$  { odd }	& $= 2$ &\ {   \cite{Bose1947} }  \\  \hline
			$s= 0,\, q>2$ { even }	& $= 3$ & \ {  \cite{MR704102} } \\  \hline
			$s = 1, q>3$	& $= 2$ & {\cite{MR0083141},\cite{MR1466573}, \cite{MR0377682}} \\  \hline
			$q$  { odd }, $1\le s < q-2$	& $=2$ & { \cite{MR0083141} }\\  \hline
			$q$  { even }, $(s+2)\nmid q, 1\le s < q-2$	& $=2$ & { \cite{MR0083141} }\\  \hline
			$q$  { even }, $(s+2)\mid q, 1\le s \le q-2$	& $\ge 3$ &\ {  \cite{MR0239991} }\\  \hline
			$s \ge q$	& $=2$ &  {Cor. \ref{cor: big s short code}} \\  \hline
			$s=q-1, q\ne 2$	& $=3$ &   {Cor. \ref{cor: q-1 and q-2}} \\  \hline
			$s= q-1, q=2$	& $= 4$ &  {Cor. \ref{cor: q-1 and q-2}} \\ \hline
			$s=q-2, q>3$	& $=4$ &  {Cor. \ref{cor: q-1 and q-2}} \\ \hline
		\end{tabular}
	\end{table}

	From the bounds in Corollary \ref{cor: q-1 and q-2}, and by  computing the parameters in Lemmas \ref{lem: full length comb} and \ref{lem: alpha beta} for some small values of $s$ and $q$, we may leverage the bounds in Table \ref{table: summary of kappa values} to obtain the values of $\kappa(s,q)$ in Table \ref{table: small kappa values}. Entries on the diagonal follow from Corollary \ref{cor: q-1 and q-2}, entries with an asterisk follow from Corollary \ref{cor: AMDS k bound}, and remaining entries were verified numerically using Sage \cite{sage}.

	\begin{table}[ht]
		\centering
		\caption{Bounds on  $\kappa = \kappa (s,q)$, $(s+2,q)=(2^a,2^b)$} \label{table: small kappa values}
		\begin{tabular}{|c||*{12}{c|}}
			\hline
			$q \backslash (s+2)$  & $2^{2}$ & $2^{3}$ & $2^{4}$ & $2^{5}$ & $2^{6}$ & $2^{7}$ & $2^{8}$ & $2^{9}$ & $2^{10}$ & $2^{11}$  & $2^{12}$ &$2^{13}$ \\
			\hline\hline
			$2$  &2 &  & & & & &  & & & && \\ \hline
			$2^2$  &4 &  & & & & &  & & & &&  \\
			\hline
			$2^3$  &3 & 4 & & & & &  & & & && \\
			\hline
			$2^4$ & 3& 3 &  4 & & & &  & && & & \\
			\hline
			$2^5$ & $\le 5$& 3 & 3 &  4 & & & & & & &&  \\
			\hline
			$2^6$ &$\le 64^*$ &3 &3 &3 &  4 & &  && & &&  \\
			\hline
			$2^7$ & 3& 3& 3& 3&3 &  4 &  & & & && \\
			\hline
			$2^8$ & 3& $\le 5$& 3& 3& 3&3 &  4 & &  & & &\\
			\hline
			$2^9$ & $\le 27$  & $\le 4$ & 3& 3& 3& 3& 3 &  4 && &  &\\
			\hline
			$2^{10}$ & $\le 5$ & 3&3 &3 & 3& 3& 3 &3 &  4 & & &\\
			\hline
			$2^{11}$ & 3 & 3&$\le 5$ &3 & 3& 3& 3 &3 &  3 & 4 & &\\
			\hline
			$2^{12}$ & 3 & 3& $\le 4$ & 3 & 3& 3& 3 &3 &  3 & 3 & 4&\\ %\le 54
			\hline
			$2^{13}$ & $\le 5$ & 3& 3 & 3 & 3& 3& 3 &3 &  3 & 3 & 3& 4\\ %\le 54
			\hline
		\end{tabular}
	\end{table}

	If $\G$ is an $[n,k,d]_q$ AMDS code with $n>q+k$, then by Theorem \ref{thm: long AMDS is NMDS}, $\G$ is NMDS. In the next section, we address a gap in the literature by providing sufficient conditions on $n$ and $k$ under which an A$^s$MDS code must be N$^s$MDS, $s\ge 1$.

	\subsection{More on N$^s$MDS codes}

	For $q>3$, if an $[2q+3,2q,3]_q$ NMDS code were to exist, then the dual code would be an $[2q+3,3,2q]_q$ code, and thus, by part \ref{part 2 cor: from Barlotti} of Corollary \ref{cor: from Barlotti}, we arrive at the contradiction $2q+3\le 2q+1$.  By the discussion in Section \ref{sec: trivial codes}, an $[2q+2,2,2q]_q$ code exists and its dual is an $[2q+2,2q,2]_q$ NMDS code.

	If $\G$ is an $[k+3,k,3]_q$ NMDS code, then $\G^{\perp}$ is an $[k+3,3,k]_q$ NMDS code, whence $k+3\le 2q+1$, so $k\le 2q-2$. From the discussion in Section \ref{sec: trivial codes}, the dual code of a $[k+2,2,k]_q$ code, $1\le k\le 2q$, is an $ [k+2,k,2]_q$ NMDS code. We thus have the following:

	\begin{lemma} \label{lem: NMDS special k}
		If $3\le k \le 2q $ then $m'(k,q)\ge k+2$ with equality if $q>3$ and $k = 2q-1,2q $.
	\end{lemma}
	\begin{remark}
		cf.\ parts \ref{part: thm Dodunekov P5} and \ref{part: thm Dodunekov P6} of Theorem \ref{thm: Dodunekov 2.7}.
	\end{remark}

	With reference to Theorem \ref{thm: ub ASMDS} we obtain the following.

	\begin{corollary}\label{cor: NMDS general bound}
		Let  $\G$ be an $[n,k,d]_q$ N$^s$MDS code.
		\begin{enumerate}
			\item If $s=1$ then  $n\le 2q +k $, and if $n\ge k+3$ then $k\le 2q-2$.
			\item If $s>1$ then  $n\le \min\{s(q+1)+k-1, s(q+1)+k^\perp-1\}$, and $k,k^\perp \le s(q+1)-1$.
		\end{enumerate}
	\end{corollary}

	The dual of an MDS code is also MDS, and if $\G$ is an $[n,k,n-k]_q$ AMDS code with $n>q+k$  then (Theorem \ref{thm: ub ASMDS}, part \ref{part: 1 thm ub ASMDS}) $\G^\perp$ is also AMDS, so $\G$ is NMDS.  For $s>1$, Liao and Liao (\cite{Liao2014}, Theorem 3.6) provide lower bounds on $n$ and $k$ sufficient for an $[n,k,d]_q$ \AS code to be N$^s$MDS, so too does Tong (\cite{Tong2012}, Theorem 4).  However, in both works the authors require $n>s(q+1)+k-1$, so according to part \ref{part: 1 thm ub ASMDS} of  Theorem \ref{thm: ub ASMDS} their results hold vacuously. This observation identifies a gap in the existing literature, which we address by establishing feasible bounds within which such codes demonstrably exist. Towards this end, we first provide the following lemma.

	\begin{lemma}\label{lem: conditions for NsMDS}
		Let $\G$ be an $[n,k,d]_q$ \AS code, $s,d>1$. If $k\ge (s-1)(q+1)$, and every $(k-s)$-subset of $\G$ is independent, then $\G$ is N$^s$MDS.
	\end{lemma}
	\begin{proof}
		Suppose the conditions hold and that $\G^\perp$ is an $[n,k^\perp,d^\perp]_q$ A$^t$MDS code.  The conditions provide $n\ge k^\perp + (s-1)(q+1)>k^\perp+s$, so (Corollary \ref{cor: bound s = 1 t>1}), $t>1$. As such, the condition $k\ge (s-1)(q+1)$ gives $t\ge s$ (Theorem \ref{thm: ub ASMDS}, part \ref{part: 1 thm ub ASMDS}). Since every $(k-s)$-subset of $\G$ is independent, it must be the case that $d^\perp (=k+1-t) \ge k-s+1$, whence  $t\le s$.
	\end{proof}

	The previous Lemma provides some easily obtained examples of N$^2$MDS binary codes.
	\begin{example}
		Let $\G$ be a binary $[n,k,d]_2$ A$^2$MDS code, where $\G^\perp$ is an \AT code.
		\begin{enumerate}
			\item If $k=3$, then (since $\G$ is assumed to be non-degenerate), $\G$ is N$^2$MDS.
			\item If $k=4$, and $\G$ is projective, then $\G$ is N$^2$MDS.
			\item If $k=5$, and $\G$ is projective and contains no line, then $\G$ is N$^2$MDS.
		\end{enumerate}
	\end{example}

	We now present some bounds providing sufficiency conditions under which an A$^s$MDS code is necessarily an N$^s$MDS code.
	\begin{theorem}\label{thm: NMDS}
		Let $\G$ be an $[n,k,d]_q$ \AS code, $1< s < q-1$, $(s+1,q) \ne (2^e,2^h)$. If $k>(s-1)(q+1)-1$ and $n>s(q+1)+k-3$ then $\G^\perp$ is an \AS code, so $\G$ is an N$^s$MDS code.
	\end{theorem}
	\begin{proof}
		By assumption, $n>q+k^\perp$ and $s>1$, so applying Corollary \ref{cor: bound s = 1 t>1} to $\G^\perp$, $t>1$. Since $n>s(q+1)+k-3$ we also have $n>k+s$, so $d>1$. Thus, by part \ref{part: 1 thm ub ASMDS}  of Theorem \ref{thm: ub ASMDS}, the condition $k>(s-1)(q+1)-1$ gives $t\ge s$.  There exists a $(d^\perp-2)$-flat $\Lambda$ with $|\Lambda\cap \G|\ge d^\perp$. Let $\Omega\subseteq \Lambda$ be an $(d^\perp-3)$-flat with $|\Omega\cap \G|=d^\perp-2$. Taking the quotient at $\Omega$ provides a non-projective $[n-d^\perp+2, k-d^\perp+2,d']_q$ A$^{s'}$MDS code, where $d'\ge d$, and therefore $s'\le s$. Note that  $\G'$ is not projective since the point corresponding to $\Lambda$ has multiplicity at least $2$. Observe that if $t>s$, then $d^\perp =k+1-t<k+1-s$, so $k-d^\perp \ge s\ge 2$, and we obtain by  part \ref{part: 3 thm proj} of Theorem \ref{thm-Projective}
		\begin{align}
			n-d^\perp+2 & \le m^{s'-1}(k-d^\perp+1,q)+2  \label{eqn: NsMDS 1}\\
			& \le m^{s-1}(k-d^\perp+1,q)+2\label{eqn: NsMDS 2}\\
			&\le  s(q+1)+k-d^\perp -1,\label{eqn: NsMDS 3}
		\end{align}
		where (\ref{eqn: NsMDS 1}) follows from part \ref{part: 3 thm proj} of Theorem \ref{thm-Projective} (if $n > m^{s'-1}(k-1,q)+2$ then $\G$ is projective),  (\ref{eqn: NsMDS 2}) follows from part \ref{part: 1 bounds on length of AsMDS} of Lemma \ref{lem: bounds on length of AsMDS} (which gives $m^{s'-1}(k,q) \le m^{s-1}(k,q)$ since $s'\le s$), and  (\ref{eqn: NsMDS 3}) follows from  part  \ref{part 2 cor: from Barlotti} of Corollary \ref{cor: from Barlotti} (which gives $m^{s-1}(k,q)\le s(q+1)+k-4$ when $(s+1,q)\ne(2^e,2^h)$). Since this contradicts $n> s(q+1)+k-3$, it must be the case that $t=s$.
	\end{proof}

	We note that there are codes meeting the condition of the previous theorem, for example, both $[13,5,7]_4$, and $[14,5,8]_4$  A$^2$MDS codes exist (see e.g. \cite{codetables.de}).

	\section{Conclusion, questions, and conjectures}

	Here, we have explored bounds on the lengths of codes with non-zero Singleton defects within the framework of projective systems.  We introduced the parameters $m^{s}(k,q)$, denoting the maximum length of a (non-degenerate) $[n,k,d]_q$ A$^s$MDS code, and  $m^{s}_t(k,q)$ denoting the maximum length of a (non-degenerate) $[n,k,d]_q$ A$^s$MDS code $\G$ such that  $\G^\perp$ is an  A$^t$MDS code. Tables \ref{table: AsMDS}, and \ref{table: AstMDS}  may help place our results in context.

	\begin{table}[!htbp]
		\centering
		\caption{Summary of Bounds on $m^{s}(k,q)$, $k\ge 3$}  \label{table: AsMDS}
		\begin{tabular}{ | p{0.45\linewidth}| c |c  | }
			\hline
			Conditions & $m^{s}(k,q)$ & Ref. \\
			\hline
			$s=q-1$; $q>2$, and $2\leq k \leq 3$, or $q=2$ and $2\leq k \leq 4$ &  \multirow{3}{*}{$= (s+1)(q+1)+k-2$} & \multirow{3}{*}{Cor. \ref{cor: q-1 and q-2}}\\
			\hhline{|-~|~|}
			\multirow{1}{*}{$s=q-2$; $2\leq k\leq 4$ and $q>2$} &  & \\
			\hline
			$k\geq 3$ & $\leq (s+1)(q+1)+k-2$ & Cor. \ref{cor: from Barlotti} \\
			\hline
			prime $p$ divides $\prod_{i=1}^{s} (q(s+1)+i)$, \linebreak $s<p<k+s$ & \multirow{2}{*}{$\leq (s+1)(q+1)+k-3$} &\multirow{6}{*}{ Cor. \ref{cor: prime divisors}}\\
			\hhline{|-|-|~|}
			prime $p$ divides $\prod^{s}_{i=1}(q(s+1)+i)$, \linebreak  $s<p<k+s-1$ & \multirow{8}{*}{$\leq (s+1)(q+1)+k-4$} & \\
			\hhline{|-|~|}
			prime $p$ divides $\prod^{s}_{i=0}(q(s+1)+i)$,  $\gcd(p,q)=1$,  $s+1<p<k+s-1$ & & \\
			\hhline{|-|~|-|}
			prime $p$ divides $\prod_{i=1}^{s} (q(s+1)+i)$, \linebreak  $s<p<k+s$, $0<s<q-2$, $(s+2)\mid q$ & & \multirow{2}{*}{Cor. \ref{cor: cor prime divisors}} \\
			\hhline{|-|~|-|}
			$k\geq 3$, $0<s< q-2$, and $(s+2,q) \neq (2^e,2^h)$ & & \multirow{2}{*}{Cor. \ref{cor: from Barlotti}}\\
			\hhline{|-|~|~|}
			$q>3$, and $s=1$ &  & \\
			\hline
			$q$ is prime, and $s+2\le (q+3)/2$ & \multirow{2}{*}{$\leq q(s+1)+k-2$} & \multirow{2}{*}{Cor. \ref{cor: from Barlotti}}\\
			\hhline{|-|~|}
			$q$ odd, $s+2$ divides $q$, $s<\frac14\sqrt{q}-2$ &  & \\
			\hline
			$s>1$ and $k>q$ & \multirow{8}{*}{$\leq s(q+1)+k-1$} & Cor. \ref{cor: AMDS k bound}\\
			\hhline{|-|~|-|}
			$s\ge q$, and $k\ge 3$ & & Cor. \ref{cor: big s short code} \\
			\hhline{|-|~|-|}
			$m(k-1,q)=q+1$, and $s>q+2-k$ & & \multirow{3}{*}{Cor. \ref{cor: MC big s short code}} \\
			\hhline{|-|~|}
			$m(k-1,q)=q+1$, $s>q+2-k$, $q$ is prime, and $3\le k \le q$ & & \\
			\hhline{|-|~|-|}
			$k\geq 3$, and $d>q^{2}$ & & \multirow{1}{*}{Cor. \ref{cor: big d means short}} \\
			\hhline{|-~|-|}
			$s=q-1$; $q>2$, and $k\geq 4$ & &  \multirow{2}{*}{Cor. \ref{cor: q-1 and q-2}}\\
			\hhline{|-~|~|}
			$s=q-2$; $k\ge 5$ and $q>3$ & & \\
			\hline
			$s\ge a (q+1)-1$, and $k\ge 3$ & \multirow{2}{*}{$\leq (s+1-a)(q+1)+ k-2+a$} & \multirow{1}{*}{Cor. \ref{cor: big s short code}} \\
			\hhline{|-|~|-|}
			$k\geq 3$, and $d>a (q^2+q)-q$, $a\geq 0$ &  & Cor. \ref{cor: big d means short} \\
			\hline
			$s=q-1$; $q=2$, and $k\geq 5$ & $ \leq k+2$ & Cor. \ref{cor: q-1 and q-2}\\
			\hline
		\end{tabular}
	\end{table}

	\begin{table}[!htbp]
		\centering
		\caption{Bounds on $m_{t}^{s}(k,q)$, $k\ge 3$, $s<n-k$} \label{table: AstMDS}
		\begin{tabular}{| p{0.45\linewidth}| c |c  | }
			\hline
			Conditions & $m_{t}^{s}(k,q)$ & Ref. \\
			\hline
			$s>1$ & $\leq t(q+1)+k^{\perp}-1$ & \multirow{7}{*}{Thm. \ref{thm: ub ASMDS}}\\
			\hhline{|--|~|}
			$t>1$ & $\leq s(q+1)+k-1$ &  \\
			\hhline{|--|~|}
			$s,t>1$ & $\leq (s+t)(q+1)-2$ &  \\
			\hhline{|--|~|}
			$t=1$ & $\leq m(k-1,q)+k^{\perp}-s+1$ & \\
			\hhline{|--|~|}
			\multirow{2}{*}{$s=1$} & $\leq (t+1)(q+1)+k^{\perp}-2$ & \\
			\hhline{|~|-|} & $\leq m(k^{\perp}-1,q)+k-t+1$ & \\
			\hhline{|-|-|~|}
			\multirow{2}{*}{$s=t=1$, $q>3$} & $\leq 2q+k-2$ &  \\
			\hhline{|~|--|}
			& $\leq (s+3)(q+1)-6$ & \multirow{2}{*}{Cor. \ref{cor: AMDS no k}} \\
			\hhline{|--|~|}
			$s>1$, $t=1$, $q>3$ & $\leq (s+2)(q+1)-3$ & \\
			\hline
			$s>1$ and $t=1$ & $\leq q+k^{\perp}$ & \multirow{1}{*}{Cor. \ref{cor: AMDS k bound}} \\
			\hhline{|--|-|}
			$t>1$, and $d>a (q^2+q)-2q$, $a\ge 0$ & $\leq (s+1-a)(q+1)+k-2+a$ & Cor. \ref{cor: big d means short t>1} \\
			\hline
			$t>1$ and $s=1$ & $\leq m(t,q)+k-t+1$ & Cor. \ref{cor: bound s = 1 t>1} \\
			\hline
			$t>2$ and $s>0$ & $\leq m^{s-1}(3,q)+k-2$ & Cor. \ref{cor: bound s = 1 t>1} \\
			\hline
			$t>1$, $s=1$, and $q$ is an odd prime & $\leq q+k-(t-2)$ & Rem. \ref{rem: MDS Conj and AMDS t>1}\\
			\hline
			$t>1$, and $s>0$ & $\leq m^{s-1}(t,q)+k-t+1$ & Lem. \ref{lem: bounds based on d perp} \\
			\hline
			$s\in \{0,1\}$, and $k\geq (t+1)(q+1)-1$ & $\leq k+1$ &  \multirow{2}{*}{Cor. \ref{cor: big k, k+s}} \\
			\hhline{|--|~|}
			$s>1$, and $k\geq t(q+1)$ & $\leq k+s$ & \\
			\hline
			$s=t=1$, $q>3$,  and $k=2q-1, 2q$ & $\leq k+2$ &   {Lem. \ref{lem: NMDS special k}} \\
			\hline
			$t=s=1$ & $\leq 2q+k$ &  {Cor.\ref{cor: NMDS general bound}} \\
			\hline
		\end{tabular}
	\end{table}

	Our results provide new insights that both extend and refine existing knowledge in the literature, with the bounds obtained through these methods often subsuming those previously established.

	The findings presented highlight the effectiveness of interpreting linear codes as projective systems. This perspective not only allows us to derive tighter bounds on code lengths but also to gain insight into the structural properties of these codes. For example, we easily see that for fixed Singleton defect, dimension, and field, codes of reasonable length must meet the Griesmer bound, must be projective, and must be dual to an AMDS code.

	The results and methods presented here have potential implications for the design and analysis of error-correcting codes, particularly in scenarios where maximizing code length is critical.  Based on the findings and observations, we propose the following questions and conjectures for further investigation.

	\newpage
	\noindent \textbf{Question 1:}  Codes meeting the hypothesis of Theorem \ref{thm: NMDS} form an extremely restricted family (especially in view of Corollary \ref{cor: bound s = 1 t>1}).  Can the bounds be widened, or otherwise relaxed to include more codes?

	\noindent \textbf{Question 2:} For which values of $k,s,$ and $q$ is the inequality in part \ref{part: 2 bounds on length of AsMDS} of Lemma \ref{lem: bounds on length of AsMDS} strict? For example, when $s=0$ and $k = 3$, the inequality is strict only if $q$ is odd.

	The dual of the perfect ternary Golay code is a projective two-weight $[11,5,6]_3$ code underlying the Berlekamp–van Lint–Seidel strongly regular graph $\mathrm{SRG}(243,22,1,2)$.
	The $[11,5,6]_3$ code dual to the perfect ternary Golay code is a length-maximal two-weight code with weights $6$ and $9$. The self-dual extended ternary Golay code $[12,6,6]_3$ is a three-weight length-maximal code, with weights $6$, $9$ and $12$. We suspect that, up to equivalence, these are the only examples of linear length maximal codes in dimensions $k>4$.

	\begin{conjecture} (Weak $\kappa(s,q)$ conjecture)
		Linear length-maximal codes of dimension $5$ or more do not exist if $q>3$. That is:  $\kappa(s,q)\le 4$ for $q>3 $.
	\end{conjecture}

	\begin{conjecture} (Strong $\kappa(s,q)$ conjecture)
		For $q>3$, linear length-maximal codes of dimension $5$ or more do not exist, moreover
		\begin{equation*}
			\kappa(s,q) = \left\{ \begin{array}{cl}
				2	& \text{ if $s\ne q-2$, and either $q$ is odd, or if $q$ is even and } (s+2)\nmid q;  \\
				3	&  \text{ if $s\ne q-2$, $q$ even, and $(s+2)\mid q$ },\text{ or } s=0\text{ and } q=2;   \\
				4	& \text{ if } s=q-2,\; q>3. %\\
			\end{array} \right.
		\end{equation*}

	\end{conjecture}

	\providecommand{\bysame}{\leavevmode\hbox to3em{\hrulefill}\thinspace}
	\providecommand{\MR}{\relax\ifhmode\unskip\space\fi MR }
	\providecommand{\MRhref}[2]{%
		\href{http://www.ams.org/mathscinet-getitem?mr=#1}{#2}
	}
	\providecommand{\href}[2]{#2}

\end{document}